\documentclass[authoryear,11pt]{elsarticle}
\journal{Statistics and Probability Letters}

\usepackage{amssymb,amsmath,multirow}
\usepackage{setspace}
\newcommand{\tr}{ {\rm tr} }
\newcommand{\T}{\top}
\newcommand{\bm}[1]{\boldsymbol{#1}}

\newtheorem{proposition}{Proposition}
\newproof{proof}{\bf Proof}

\begin{document}
	%
	
	\begin{frontmatter}
		\title{ A unified approach for covariance matrix estimation under Stein loss}
		\cortext[cor1]{Corresponding author}
		\author{Anis M. Haddouche\corref{cor1}\fnref{label1}}
		\author{Wei Lu \fnref{label2}}
		\affiliation[label1]{organization={INSA Rouen, Normandie Univ, UNIROUEN, UNIHAVRE, INSA Rouen,  LITIS and LMI},
			addressline={avenue de l'Universit\'e, BP 8},
			city={Saint-\'Etienne-du-Rouvray},
			postcode={76801},
			country={France.},
			ead={ Mohamed.haddouche@insa-rouen.fr}}
		\affiliation[label2]{organization={INSA Rouen, Normandie Univ, UNIROUEN, UNIHAVRE, INSA Rouen, LMI},
			addressline={avenue de l'Universit\'e, BP 8},
			city={Saint-\'Etienne-du-Rouvray},
			postcode={76801},
			country={France.},
			ead={ wei.lu@insa-rouen.fr}}

\begin{abstract}
	In this paper, we address the problem of estimating a covariance matrix of a multivariate Gaussian distribution, relative to a Stein loss function, from a decision theoretic point of view. We investigate the case  where the  covariance matrix is invertible and the case when it is non--invertible in a unified approach.	
\end{abstract}

\end{frontmatter}

\section{Introduction}\label{Introduction}


Let $\bm{X}$ be an observed $p \times n$  matrix of the form
\begin{align}\label{model}
	\bm{X}=\bm{B}\bm{Z},
\end{align}
where $\bm{B}$ is a $p \times r$ matrix of unknown  parameters, with $r\leq p$, and $\bm{Z}$ is a $r \times n$ random matrix.
Assume that $r$ is known and that the columns of $\bm{Z}$ are identically and independently distributed as the $r$-dimensional multivariate normal distribution $\mathcal N_r({0}_r,\bm{I}_r)$.
Then the columns of $\bm{X}$ are identically and independently distributed from the $p$-dimensional multivariate normal  $\mathcal N_p({0}_p,\bm{\Sigma})$, where $\bm{\Sigma} = \bm{B}\bm{B}^T$ is the unknown $p \times p$  covariance matrix with
\[ {\rm rank}(\bm{\Sigma}) =r \leq p\,.\]
It follows that the $p\times p$ sample covariance matrix $\bm{S} = \bm{X}\bm{X}^T$ has  a singular Wishart distribution (see \cite{Srivastava2003}) such that
\[{\rm rank}(\bm{S})=\min(n,r)= q\leq p\,, \]
with probability one.
We denote in the following by $\bm{S}^{+}$ and $\bm{\Sigma}^{+}$ the Moore-Penrose inverses of $\bm{S}$ and $\bm{\Sigma}$ respectively.

We consider the problem of estimating the covariance matrix $\bm{\Sigma}$ under the Stein type loss function 
\begin{align}\label{Loss_Stein}
	L(\bm{\hat \Sigma},\bm{\Sigma})=\tr({\bm{\Sigma^{+}{\hat \Sigma}}}) - \ln \!\vert \Lambda(\bm{\Sigma^{+}{\hat \Sigma}} )\vert  - q\,,
\end{align}
where $\bm{\hat \Sigma}$ estimates $\bm{\Sigma}$ and $ \Lambda(\bm{\Sigma^{+}{\hat \Sigma}} )$ is the diagonal matrix of  the $q$ positives  eigenvalues of $\bm{\Sigma^{+}{\hat \Sigma}}$. The corresponding risk function is denoted by
\begin{align*}
	R(\bm{\hat \Sigma},\bm{\Sigma})= {E}[L(\bm{\hat \Sigma},\bm{\Sigma})]\,,
\end{align*}
where ${E}(\cdot)$ denotes the expectation with respect to the model \eqref{model}.
Note that the loss function \eqref{Loss_Stein} is an adaptation of the original Stein loss function (see \cite{Stein1977}) to the context of the model \eqref{model} (see \cite{tsukuma2016} for more details).

%
%

The difficulty of covariance estimation is commonly characterized by the ratio $ p/n$. 
The  usual estimators of the form 
\begin{align}\label{Natural_estimators}
	\bm{\hat \Sigma}_a={a}\,\bm{S},\quad \text{with}\quad a>0,
\end{align}
perform poorly when $n,p \to \infty$ with $p/n \to c >0$ (see \cite{Ledoit2004}). Hence, in this situation, alternative estimators are needed. Indeed,
as pointed out by  \cite{James1961a}, the larger (smaller) eigenvalues of $\bm\Sigma$ are overestimated (underestimated) by those estimators. Therefore, a possible approach to derive an improved  estimators is to regularize the eigenvalues of $\bm{\hat \Sigma}_{a}$. 
This fact suggest to consider the class of orthogonally invariant estimators (see \cite{Takemura1984}) in \eqref{O.I.V.estimators} below.

Considering the model \eqref{model}, we deal, in a unified approach, with the following cases.
\begin{enumerate}[$(i)$]
	\item $n<r=p$: $\bm{\Sigma}$ is invertible of ${\rm rank}$ $p$ and $\bm{S}$ is non--invertible of  ${\rm rank}$ $n$;
	\item $r=p\le n$: $\bm{\Sigma}$ and  $\bm{S}$ are invertible;
	\item $r<p\le n$: $\bm{\Sigma}$ and  $\bm{S}$ are non--invertible of rank $r$;
	\item $r\le n<p$: $\bm{\Sigma}$ and  $\bm{S}$ are non--invertible of rank $r$;
	\item $n<r<p$: $\bm{\Sigma}$ and   $\bm{S}$ are  non--invertible of ranks $r$ and $n$ respectively.
\end{enumerate}
The class of orthogonally invariant estimators was considered by various authors. For instance, see  \cite{Stein1977}, \cite{dey1985estimation} and \cite{Haff1980} for the case $(i)$, \cite{Konno2009} and \cite{HADDOUCHE2021104680} for the cases $(i)$ and $(ii)$. See also \cite{ChetelatWells2016} for the cases $(iii)$ and $(iv)$. Recently \cite{tsukuma2016} extend the \cite{Stein1977} estimator to the five possible cases above in a unified approach. Similarly, we extend here the class of \cite{Haff1980}  estimators to the context of the model \eqref{model}.

The rest of this paper is organized as follows. In Section \ref{main.result}, we  derive the improvement result of the proposed estimators over the usual estimators. We study  numerically the behavior of the proposed estimators in Section \ref{numerical.study}.

\section{Main result}\label{main.result}

Improving the class of the natural estimators in \eqref{Natural_estimators} relies on improving the optimal estimator among this class, that is, the one which minimizes the loss function \eqref{Loss_Stein}.
\begin{proposition}[\cite{tsukuma2016}] Under the Stein loss function \eqref{Loss_Stein}, the optimal estimator among the class \eqref{Natural_estimators} 
 is given by
	\begin{align}\label{Optimal_Const_Stein}
		\bm{\hat \Sigma}_{a_{o}}={a_{o}}\bm{S},\quad \text{where}\quad a_o=\frac{1}{m} 
		\quad 
			\text{and} 
			\quad
			m=\max{(n,r)}.
	\end{align}
\end{proposition}
%

%
As mentioned in Section \ref{Introduction}, we consider the class of orthogonnally invariant estimators. Let $\bm{S}=\bm{H}\,\bm{L}\,\bm{H}^{\top}$  be the eigenvalue decomposition of $\bm{S}$
where $\bm{H}$ is a $p\times q $ semi--orthogonal matrix of eigenvectors and $\bm{L}={\rm diag}(l_1,\dots,l_q)$, with $l_1 >,\dots,>l_q$, is the diagonal matrix of the $q$ positive corresponding eigenvalues (see  \cite{KubokawaSrivastava2008} for more details). The class  of orthogonally invariant estimators is of the form
\begin{align}\label{O.I.V.estimators}
	\hat{\bm \Sigma}_{\Psi}
	&= a_{o}\,\big(\bm{S} + \bm{H}\,\bm{L}\,\bm{\Psi}(\bm{L})\,\bm{H}^{\T})  \,
\end{align}
with  $\bm{\Psi}(\bm{L})={\rm diag}(\psi_1(\bm{L}),\dots,\psi_q(\bm{L}))$, where $\psi_i(\bm{L})$ ($i=1,\dots,q$) is a differentiable function of $\bm{L}$.

More precisely, we consider an extension of the class of \cite{Haff1980}  estimators, to the context of the model \eqref{model}, defined as
\begin{align}\label{haff.estimators}
	\hat{\bm{\Sigma}}_{\alpha} = a_{o}\big(\bm{S} + \bm{H}\bm{L}\bm{\Psi}(\bm{L})\bm{H}^{\T}\big) 
	%
	\, \text{with }\,\, \alpha \geq 1 \,,\, b >0
	\,\, \text{and}  \,\,
	\,\bm{\Psi}(\bm{L}) = b\frac{\bm{L}^{-\alpha}}{\tr (\bm{L}^{-\alpha})},
\end{align}
where $a_{o}$ is given in \eqref{Optimal_Const_Stein}.
We give in the following proposition our main result.

\begin{proposition}\label{proposition}
The Haff type estimators in \eqref{haff.estimators} improves over the optimal estimator in \eqref{Optimal_Const_Stein}, under the loss function \eqref{Loss_Stein}, as soon as 
\begin{align*}
	  0 < b \leq  b_o= \frac{2\,(q-1)}{ m- q +1 }\,.
\end{align*}
\end{proposition}
\begin{proof}
We aim to show that the risk difference between the Haff type estimators in \eqref{haff.estimators} and  the optimal estimator in \eqref{Optimal_Const_Stein}, namely,
\begin{align}\label{risk.diff}
	\Delta_{(\alpha,a_o)} =R(\bm{\hat \Sigma}_\alpha,\bm{\Sigma}) - R(\bm{\hat \Sigma}_{a_o},\bm{\Sigma}),
\end{align}
is non--positive.
Note that $\bm{\hat \Sigma}_\alpha$ can be written as
\begin{align*}
\hat{\bm{\Sigma}}_{\alpha}= a_{o}\,\bm{H}\bm{L}\bm{\Phi}(\bm{L})\bm{H}^{\T}  \quad \text{with} \quad \bm{\Phi}(\bm{L})= \bm{I}_q + \bm{\Psi}(\bm{L}).
\end{align*}
The risk of these estimators under the Stein loss function  \eqref{Loss_Stein} is given by 
\begin{align}\label{risk.alternative.0}
	R(\hat{\bm{\Sigma}}_{\alpha},\bm{\Sigma}) = {E}\big(\tr(\bm{\Sigma}^{+} \hat{\bm{\Sigma}}_{\alpha})\big)
	 -{E}\big(\!\ln\!\vert \Lambda (\bm{\Sigma}^{+}\hat{\bm{\Sigma}}_{\alpha})\vert\big) - q\,.
\end{align} 

First, dealing with ${E}\big(\tr(\hat{\bm{\Sigma}}_{\alpha}\bm{\Sigma}^{+})\big)$, we apply Lemma A.2 in \cite{tsukuma2016} in order to get rid of the unknown parameter $\bm{\Sigma}^{+}$. It follows that,
\begin{align}\label{risk.alternative.1}
	{E}\big(\tr(\bm{\Sigma}^{+}\hat{\bm{\Sigma}}_{\alpha})\big) 
%
	&=
	a_o{E}\left(\sum_{i=1}^q \!\left\{\!(m-q+1)\varphi_i+2l_i\frac{\partial\varphi_i}{\partial l_i}+2\sum_{j> i}^q\frac{l_i\varphi_i-l_j\varphi_j}{l_i - l_j}\!\right\}\!\right),
\end{align}
where,  for $i=1,\dots,q$, 
\begin{align*}
\phi_i= 1+ b\, \frac{l_i^{-\alpha}} {\tr(L^{-\alpha} )},
\quad 
\frac{\partial\varphi_i}{\partial l_i} = b\,\alpha\,\frac{1-\tr(\bm{L}^{-\alpha})\,l_i^{\alpha}}{\tr^2(\bm{L}^{-\alpha})\,l_i^{1+2\alpha}}
\intertext{and}
 \sum_{j> i}^q
\frac{l_i\varphi_i-l_j\varphi_j}{l_i - l_j}
=\sum_{j> i}^q \left(1+\frac{b}{\tr(\bm{L}^{-\alpha})}\left\{\frac{l_i^{1-\alpha}-l_j^{1-\alpha}}{l_i-l_j} \right\} \right)\,.
\end{align*}
Using the fact, for $j>i$, $l_j > l_i$, it can be shown that 
\begin{align}\label{maj.0}
	 \sum_{j> i}^q
	\frac{l_i\varphi_i-l_j\varphi_j}{l_i - l_i}
	 \leq (q-i)\,.
\end{align}
Therefore, using \eqref{maj.0}, we obtain 
%
\begin{align*}
	{E}(\tr(\bm{\Sigma}^{+}\hat{\bm{\Sigma}}_{\alpha}))
	&\leq
	a_{o}\,{E}\Biggr(\sum_{i=1}^q\Biggr\{(m-q+1)\left(1+b\frac{{l_i}^{-\alpha}}{\tr(\bm{L}^{-\alpha})}\right) \nonumber \\
	&\hspace{2.7cm}
	+2\,b\,\alpha\,\frac{1-\tr(\bm{L}^{-\alpha})\,l_i^{\alpha}}{\tr^2(\bm{L}^{-\alpha})\,l_i^{2\alpha}}+2(q-i)\Biggl\}\Biggl) \nonumber \\
 	&=
   a_o\,m\,q+a_o\,b\,{E}\Biggr(\sum_{i=1}^q\Biggr\{(m-q+1)\frac{{l_i}^{-\alpha}}{\tr(\bm{L}^{-\alpha})}
   \nonumber \\
   &\hspace{5cm}
   +2\,\alpha\frac{1-\tr(\bm{L}^{-\alpha})\,l_i^{\alpha}}{\tr^2(\bm{L}^{-\alpha})\,l_i^{2\alpha}}\Biggl\}\Biggl)
    \nonumber \\
   &=
   a_{o}\left(m\,q+b\,(m-q+1)\right) + 2\,\alpha\,{E} \left(\frac{\tr(\bm{L}^{-2\alpha})}{\tr^2(\bm{L}^{-\alpha})}-1 \right)\,.
\end{align*}
From the submultiplicativity of the trace for semi--definite positive matrices, we have  $\tr(\bm{L}^{-2\alpha})\leq \tr^{2}(\bm{L}^{-\alpha})$. Then, an upper bound for \eqref{risk.alternative.1} is given by
%
\begin{align}\label{risk.alternative.2.bis}
	{E}(\tr(\bm{\Sigma}^{+}\hat{\bm{\Sigma}}_{\alpha}))
	&\leq
	 a_{o}\left(m\,q+b\,(m-q+1)\right)\,.
\end{align}

Secondly, dealing with ${E}\big(\!\ln\vert \Lambda (\hat{\bm{\Sigma}}_{\alpha}\bm{\Sigma}^{+})\vert\big)$ in \eqref{risk.alternative.0}, it can be shown that  
\begin{align*}
\Lambda(\bm{\Sigma}^{+}\hat{\bm{\Sigma}}_{\alpha}) 
= a_{o}\,\Lambda\bigr(\bm{\Sigma}^{+}\bm{H}\bm{L}\bm{\Phi}(\bm{L})\bm{H^{\T}}	\bigl) 
=
a_{o}\,\Lambda\bigr(\bm{L}^{1/2} \bm{H^{\T}} \bm{\Sigma^{+}} \bm{H} \bm{L}^{1/2} \bm{\Phi}(\bm{L})	\bigl).
\end{align*}
Note that  $\bm{L}^{1/2} \bm{H^{\T}} \bm{\Sigma^{+}} \bm{H} \bm{L}^{1/2}$ and $\bm{\Phi}(\bm{L})$ are full rank $q\times q$ matrices. It follows that $$\vert\Lambda(\hat{\bm{\Sigma}}_{\alpha}\bm{\Sigma}^{+})\vert =  a^q_{o}\,\vert\bm{L}^{1/2} \bm{H^{\T}} \bm{\Sigma^{+}} \bm{H} \bm{L}^{1/2} \bm{\Phi}(\bm{L})	\vert .$$ 
Therefore 
\begin{align}\label{risk.alternative4}
	{E}(\ln\vert \Lambda (\hat{\bm{\Sigma}}_{\alpha}\bm{\Sigma}^{+})\vert) 
	&=q\,\ln(a_o)+{E}(\ln\vert \bm{L}^{1/2}\bm{H}^T\bm{\Sigma}^{+}\bm{H}\bm{L}^{1/2}\vert)
	+{E}(\ln\vert\bm{\Phi}(\bm{L})\vert).
\end{align}
Using the fact that  $\ln(1+x) \geq 2\,x / (2+x)$, for any positive constant $x$, then 
\begin{align*}
	\ln\vert \bm{\Phi}(\bm{L}) \vert  
	= \ln \vert \bm{I}_q + \frac{b\,\bm{L}^{-\alpha}}{\tr(\bm{L}^{-\alpha})}\vert  
%
	&=\sum_{i=1}^{q}\ln \left( 1 +  \frac{b\,l^{-\alpha}_i}{\tr(\bm{L}^{-\alpha})}\right) 
	\\
	&\geq
	\sum_{i=1}^{q} \frac{2\,b\,l^{-\alpha}_i / \tr(\bm{L}^{-\alpha} )}{2+ b\,l^{-\alpha}_i / \tr(\bm{L}^{-\alpha} )}.
\end{align*} 
Thus
\begin{align}\label{maj.1}
\ln\vert \bm{\Phi}(\bm{L}) \vert  \geq \frac{2b}{2+b}\,,
\end{align}
since, for $i=1,\dots,q$, $l^{-\alpha}_i \leq \tr(\bm{L}^{-\alpha})$. Consequently, thanks to \eqref{maj.1}, a lower bound for \eqref{risk.alternative4} is given by
\begin{align}\label{risk.alternative3}
{E}(\ln\vert \Lambda (\hat{\bm{\Sigma}}_{\alpha}\bm{\Sigma}^{+})\vert) \ge q\,\ln(a_o)+{E}(\ln\vert \bm{L}^{1/2}\bm{H}^T\bm{\Sigma}^{+}\bm{H}\bm{L}^{1/2}\vert)+\frac{2\,b}{2+b}.
\end{align}

Now, relying on  the proof of Proposition 2.1 in \cite{tsukuma2016},  it can be shown that
\begin{align}\label{risk.aS.1}
	R(\hat{\bm{\Sigma}}_{a_o}, \bm{\Sigma}) = -q\ln(a_o) -  {E}(\ln\!\vert \bm{L}^{1/2}\bm{H}^T\bm{\Sigma}^{+}\bm{H}\bm{L}^{1/2}\vert).
	\end{align}

Finally, combining \eqref{risk.alternative.2.bis}, \eqref{risk.alternative3} and \eqref{risk.aS.1}, an upper bound for the risk difference in \eqref{risk.diff} is given by 
\begin{align*}
	\Delta_{(\alpha,a_o)}
	&\leq
	 a_o\,(m-q+1)\,b - \frac{2\,b}{2+b} 
	=b\,\Bigr(a_o\,(m+q-1)\,(b+2) - 2 \Bigl),
\end{align*}
since $a_{o}=1/m$, which  is non-positive as soon as 
\begin{align*}
	0 < b \leq b_o= \frac{2\,(q-1)}{m-q+1},
\end{align*}
\qed
\end{proof}

\section{Numerical study}\label{numerical.study}
We study here numerically the performance of the proposed estimators of the form
\begin{align}\label{haff.simulation}
	\hat{\bm{\Sigma}}_{\alpha} = a_{o}\big(\bm{S} + \bm{H}\bm{L}\bm{\Psi}(\bm{L})\bm{H}^{\T}\big) 
	\quad 
	\text{with}
	\quad
	\alpha >0,
	\quad
	\bm{\Psi}(\bm{L}) = b_o\frac{\bm{L}^{-\alpha}}{\tr (\bm{L}^{-\alpha})},
\end{align}
where $b_o$ is given in Proposition \ref{proposition}.

We consider the following structures of $\bm{\Sigma}$: $(i)$ the identity matrix $\bm{I}_p$ and $(ii)$ an autoregressive structure with coefficient $0.9$. We set their $p-r$ smallest eigenvalues  to zero in order to construct  matrices of rank $r\leq p$.

To assess the performance of the proposed estimators, we compute  the Percentage Reduction In Average Loss (PRIAL), for some values of $p$, $n$, $r$ and $\alpha$, defined as 
\begin{align*}
	\textrm{PRIAL}(\hat{\bm{\Sigma}}_{\alpha}) = \frac{R(\hat{\bm{\Sigma}}_{a_o},\bm{\Sigma})-{R(\hat{\bm{\Sigma}}_{\alpha},\bm{\Sigma})}}{R(\hat{\bm{\Sigma}}_{a_o},\bm{\Sigma})} \times  100,
\end{align*}
where $\hat{\bm{\Sigma}}_{a_o}$ and $\hat{\bm{\Sigma}}_{\alpha}$ are respectively defined in \eqref{Optimal_Const_Stein} and \eqref{haff.simulation}.

\begin{table}[!htb]
	\centering
	\begin{tabular}{llllllll} 
		$\bm{\Sigma}$                  & $(p,n)$                     & $r$   & $\hat{\bm{\Sigma}}_1$     & $\hat{\bm{\Sigma}}_2$     & $\hat{\bm{\Sigma}}_3$     & $\hat{\bm{\Sigma}}_4$     & $\hat{\bm{\Sigma}}_5$      \\ 
		\hline\hline
		\multirow{10}{*}{$(i)$}  & \multirow{3}{*}{$(30,50)$}  & 10  & 6.85  & 12.45 & 15.53 & 16.95 & 17.53  \\
		&                           & 20  & 9.20  & 13.91 & 14.88 & 14.68 & 12.47  \\
		&                           & 30  & 11.81 & 14.33 & 13.41 & 12.43 & 11.71  \\ 
		\cline{2-8}
		& \multirow{3}{*}{$(50,30)$}  & 20  & 18.31 & 19.65 & 17.75 & 16.44 & 15.63  \\
		&                           & 40  & 17.12 & 16.33 & 14.07 & 12.78 & 12.02  \\
		&                           & 50  & 11.80 & 14.23 & 13.29 & 12.30 & 11.59  \\ 
		\cline{2-8}
		& \multirow{4}{*}{$(150,30)$} & 20  & 18.17 & 19.69 & 17.87 & 16.56 & 15.71  \\
		&                           & 40  & 17.08 & 16.31 & 14.07 & 12.76 & 11.99  \\
		&                           & 60  & 8.88  & 12.27 & 12.33 & 11.75 & 11.19  \\
		&                           & 150 & 2.83  & 5.09  & 6.50  & 7.25  & 7.61   \\ 
		\cline{2-8}
        \hline
		\multirow{10}{*}{$(ii)$} & \multirow{3}{*}{$(30,50)$}  & 10  & 6.06  & 8.70  & 9.62  & 9.89  & 9.92   \\
		&                           & 20  & 8.81  & 11.66 & 12.12 & 11.93 & 11.61  \\
		&                           & 30  & 11.46 & 13.15 & 12.35 & 11.48 & 10.82  \\ 
		\cline{2-8}
		& \multirow{3}{*}{$(50,30)$}  & 20  & 17.18 & 17.45 & 15.85 & 14.76 & 14.07  \\
		&                           & 40  & 16.34 & 15.18 & 13.19 & 12.00 & 11.28  \\
		&                           & 50  & 11.33 & 12.82 & 12.02 & 11.19 & 10.57  \\ 
		\cline{2-8}
		& \multirow{4}{*}{$(150,30)$} & 20  & 17.30 & 18.00 & 16.38 & 15.19 & 14.42  \\
		&                           & 40  & 16.27 & 15.04 & 13.04 & 11.84 & 11.13  \\
		&                           & 60  & 8.48  & 10.43 & 10.29 & 9.81  & 9.35   \\
		&                           & 150 & 2.73  & 4.03  & 4.61  & 4.86  & 4.95   \\ 
		\cline{2-8}
		&                           &     &       &       &       &       &       
	\end{tabular}
	\caption{Effect of $\alpha=1,\dots,5$ on PRIAL's for the structures $(i)$ and $(ii)$ of $\bm \Sigma$.}
	\label{tab 1}
\end{table}

Table \ref{tab 1} shows that the proposed estimators improve over $\hat{\bm{\Sigma}}_{a_o}$ for any possible ordering of $p,n$ and $r$.  Compared to other cases,   the Haff type estimators $\hat{\bm{\Sigma}}_{\alpha}$ (for $\alpha=1,\dots,5$) have  better performances in the case where $p>n>r$, with PRIAL's higher than $14.07 \%$ for both structures $(i)$ and $(ii)$ of $\bm\Sigma$.
We report that  the optimal value of $\alpha$, which maximizes the PRIAL's, depends on $p,n$ and $r$.
%

\newpage

\section*{Acknowledgement}
This research did not receive any specific grant from funding agencies in the public, commercial, or not-for-profit sectors.

\bibliographystyle{elsarticle-num-names} 
\bibliography{Anis_Biblio_2021.bib}
\end{document}